\documentclass[10pt]{amsart}
\usepackage{amssymb}
\usepackage{epsfig}
\usepackage{url}
\usepackage{setspace}
\theoremstyle{plain}

\newtheorem{thm}{Theorem}[section]
\newtheorem{cor}[thm]{Corollary}

\newtheorem{rem}[thm]{Remark}
\newtheorem{ques}[thm]{Question}
\newtheorem{conj}[thm]{Conjecture}
\newtheorem{exam}[thm]{Example}
\def\cal{\mathcal}
\def\bbb{\mathbb}
\def\op{\operatorname}
\renewcommand{\phi}{\varphi}
\newcommand{\R}{\bbb{R}}
\newcommand{\N}{\bbb{N}}
\newcommand{\Z}{\bbb{Z}}
\newcommand{\Q}{\bbb{Q}}

\newcommand{\C}{\bbb{C}}

\begin{document}

\title[Rational points in arithmetic progression on $y^2=x^n+k$]{Rational points in arithmetic progression \\on $y^2=x^n+k$}
\author{Maciej Ulas}


\begin{abstract}
Let $C$ be a hyperelliptic curve given by the equation $y^2=f(x)$,
where $f\in\Z[x]$ and $f$ hasn't multiple roots. We say that points
$P_{i}=(x_{i},\;y_{i})\in C(\Q)$ for $i=1,2,\ldots,\;n$ are in
arithmetic progression if the numbers $x_{i}$ for $i=1,2,\ldots,\;n$
are in arithmetic progression.

In this paper we show that there exists a polynomial $k\in\Z[t]$
with such a property that on the elliptic curve $\cal{E}:
y^2=x^3+k(t)$ (defined over the field $\Q(t)$) we can find four
points in arithmetic progression which are independent in the group
of all $\Q(t)$-rational points on the curve $\cal{E}$. In particular
this result generalizes some earlier results of Lee and V\'{e}lez
from \cite{LeeVel}. We also show that if $n\in\N$ is odd then there
are infinitely many $k$'s with such a property that on the curves
$y^2=x^n+k$ there are four rational points in arithmetic
progressions. In the case when $n$ is even we can find infinitely
many $k$'s such that on the curves $y^2=x^n+k$ there are six
rational points in arithmetic progression.
\end{abstract}

\maketitle

\section{Introduction}\label{sec1}

Many problems in number theory are equivalent to the problem of
solution of certain equation or system of equations in integers or
in rational numbers. Problems of this type are called {\it
diophantine problems}. In the case when we are able to show that our
problem has infinitely many solutions a natural question arises
whether it is possible to show the existence of rational parametric
solutions i.e. solutions in polynomials or in rational functions. In
general this kind of problems are difficult and we don't have any
general theory which can give even partially answer to such kind
questions. For example, N. Elkies showed in \cite{Elk} that the set
of rational points on the surface $x^4+y^4+z^4=t^4$ is dense in the
set of all real points on this surface. However we still haven't
known if this equation has rational parametric solutions.

In this paper we meet with the problem of similar type. Our question
is related to the construction of integers $k$ with such a property
that on the elliptic curve $E_{k}:\; y^2=x^3+k$ there are four
rational points in arithmetic progression. Let us recall that if the
curve $C:\;f(x,y)=0$ is defined over $\Q$ and we have the rational
points $P_{i}=(x_{i},\;y_{i})$ on $C$ then we say that $P_{i}$'s are
in arithmetic progression if the numbers $x_{i}$ for
$i=1,2,\ldots,\;n$ are in arithmetic progression.

In connection with this problem Lee and V\'{e}lez in \cite{LeeVel}
showed that  each rational point on the elliptic curve
\begin{center}
$E:\;y^2=x^3-39x-173$
\end{center}
gives an integer $k$ with such a property that on the elliptic curve
$E_{k}:\;y^2=x^3+k$ there are four rational points in arithmetic
progression. Due to the fact that the set $E(\Q)$ of all rational
points on $E$ is generated by the point $(11,27)$ of infinite order
we get infinitely many $k$'s which satisfy demanded conditions.

Now, it should be noted that S. P. Mohanty stated the following

\begin{conj}[S. P. Mohanty \cite{Moh}]\label{MohantyConj}
Let $k\in\Z$ and suppose that the rational points
$P_{i}=(x_{i},y_{i})$ for $i=1,\;\ldots,\;n$ are in arithmetic
progression on the elliptic curve $y^2=x^3+k$. Then $n\leq 4$.
\end{conj}

We think that the above conjecture is not true. It is clear that in
order to find counterexample we should have plenty of $k$'s for
which on the curve $E_{k}:\;y^2=x^3+k$ we have four points in
arithmetic progression. Integer number which satisfy this condition
will be called {\it number of AP4 type}.

Main aim of this paper is to construct parametric families of
numbers of AP4 type. We should note that the method employed in
\cite{LeeVel} cannot be used in order to find families of these
kinds.

In Section \ref{sec2} we show that each rational point on the
surface
\begin{equation*}
\cal{S}_{1}:\;(p^2-3q^2+3r^2-s^2)(11p^2-18q^2+9r^2-2s^2)=3(2p^2-5q^2+4r^2-s^2)^{2}
\end{equation*}
give us an integer $k$ of AP4 type. In particular we prove that on
the $\cal{S}_{1}$ there are infinitely many rational curves. Using
this result we show that there exists a polynomial $k\in\Z[t]$ with
such a property that on the corresponding elliptic curve
$\cal{E}:\;y^2=x^3+k(t)$ defined over $\Q(t)$ there are four
$\Q(t)$-rational points and these points are independent in the set
$\cal{E}(\Q(t))$. Using this result we deduce that the set of
rational points on the surface $\cal{S}_{1}$ is dense in the set of
all real points on $\cal{S}_{1}$ in the Euclidean topology.

In Section \ref{sec3} we consider natural generalizations of the
problem of construction of rational points in arithmetic
progressions on hyperelliptic curves of the form $y^2=x^n+k$.

In Section \ref{sec4} we give a special sextic hypersurface which is
connected with the problem of construction of integer $k$ with sucha
property that on the curve $y^3=x^3+k$ there are five points in
arithmetic progression.

\section{Rational points on $\cal{S}_{1}$}\label{sec2}

In this section we are interested in the construction of integer
numbers of AP4 type. Let $f(x)=ax^3+bx^2+cx+d\in\Q[a,b,c,d][x]$ and
consider the curve $C:\;y^2=f(x)$. Using now change of coordinates
\begin{equation*}
(x,\;y)=\Big(\frac{X-3b}{9a},\;\frac{Y}{27a}\Big)\;\; \mbox{with
inverse} \quad (X,\;Y)=\Big(9ax+3b,\;27ay\Big),
\end{equation*}
we can see that the curve $C$ birationally equivalent to the curve
\begin{equation*}
E:\;Y^2=X^3+27(3ac-b^2)X+27(27a^{2}d-9abc+2b^3).
\end{equation*}

Let $p,q,r,s$ be rational parameters and let us put
\begin{equation}\label{R1}
\begin{array}{lll}
  a=-(p^2-3q^2+3r^2-s^2)/6, & &b=(2p^2-5q^2+4r^2-s^2)/2, \\
  c=-(11p^2-18q^2+9r^2-2s^2)/6, & &d=p^2.
\end{array}
\end{equation}
For $a,b,c,d$ defined in this way we have
\begin{equation*}
f(0)=p^2,\quad f(1)=q^2,\quad f(2)=r^2,\quad f(3)=s^2.
\end{equation*}
We can see that in order to prove that the set of integers of AP4
type contains image of certain polynomial we must consider the
surface $3ac=b^2$, where $a,\;b,\;c,\;d$ are define by the
(\ref{R1}).

Indeed, the points $(0,\;p),\;(1,\;q),\;(2,\;r),\;(3,\;s)$ are in
arithmetic progression on the curve $C$, and due to the fact that
our mapping from $C$ to $E$ given by (\ref{R2}) is affine we deduce
that the images of these points will be in arithmetic progression on
the curve $E$. So we consider the surface in $\mathbb{P}^{3}$ given
by the equation
\begin{equation}\label{R2}
\cal{S}_{1}:\;(p^2-3q^2+3r^2-s^2)(11p^2-18q^2+9r^2-2s^2)=3(2p^2-5q^2+4r^2-s^2)^{2}.
\end{equation}
It is easy to see that the surface $S_{1}$ is singular and that the
set $(\pm 1,\;\pm 1,\;\pm 1,\;\pm 1)$ is the set of all singular
points (the sign $+$ and $-$ are independent of each other). The
surface $S_{1}$ has sixteen singular points which is maximal number
of singular points for surfaces of degree four. Surfaces with this
property are known under the name of {\it Kummer surfaces}.

We will show the following

\begin{thm}\label{thm1}
The set of rational curve on the surface $\cal{S}_{1}$ is infinite.
In particular the set $\cal{S}_{1}(\Q)$ is dense in the set of all
real points $\cal{S}_{1}(\R)$ in the Euclidean topology.
\end{thm}
\begin{proof}
In order to show that on the surface $\cal{S}_{1}$ there are
infinitely many rational curves defined over $\Q$ let us consider
the following system of equations
\begin{equation*}
\begin{cases}
t(p^2-3q^2+3r^2-s^2)=(2p^2-5q^2+4r^2-s^2),\\
(11p^2-18q^2+9r^2-2s^2)=3t(2p^2-5q^2+4r^2-s^2),
\end{cases}
\end{equation*}
or equivalently
\begin{equation}\label{R3}
\begin{cases}
(3t^2+3t+1)r^2=-(3t^2+9t+7)p^2+2(3t^2+6t+4)q^2,\\
(3t^2+3t+1)s^2=-2(3t^2+12t+13)p^2+9(t^2+3t+3)q^2,
\end{cases}
\end{equation}
where $t$ is an indeterminate parameter. It is easy to see that each
rational solution of the system (\ref{R3}) leads us to the rational
point on the surface $\cal{S}_{1}$. From geometric point of view the
system (\ref{R3}) as intersection of two quadratic surfaces with
rational points $(p,q,r,s)=(1,1,1,1)$ is birationally equivalent
with an elliptic curve defined over the field $\Q(t)$. Now, we show
the construction of appropriate mapping.

Using standard substitution $(p,\;q,\;r)=(u+r,\;v+r,\;r)$ we can
parametrize all rational points on the first equation of the system
(\ref{R3}) in the following way
\begin{equation}\label{R4}
\begin{cases}
p=(3t^2+9t+7)u^2-4(3t^2+6t+4)uv+2(3t^2+6t+4)v^2,\\
q=(3t^2+9t+7)u^2-2(3t^2+9t+7)uv+2(3t^2+6t+4)v^2,\\
r=(3t^2+9t+7)u^2-2(3t^2+9t+4)v^2.
\end{cases}
\end{equation}
Without lose of generality we can assume that $u=1$ and let us
substitute the parametrization we have obtained into the second
equation in the system (\ref{R3}). We get the curve defined over the
field $\Q(t)$ with the equation
\begin{align*}
\cal{C}_{1}:\;s^2=\;&4(3t^2+6t+4)^2v^{4}+8(3t^2+6t+4)(3t^2+15t+19)v^{3}+\\
    &-4(36t^4+243t^3+618t^2+702t+313)v^{2}+\\
    &+4(3t^2+9t+7)(3t^2+15t+19)v+(3t^2+9t+7)^2
\end{align*}
with $\Q(t)$-rational point $Q=(0,\;3t^2+9t+7)$. Let us define the
following quantities
\begin{align*}
&C(t)=(99t^4+756t^3+2253t^2+3114t+1709)/3,\\
&D(t)=36(t^2+3t+3)(3t^2+12t+13)(3t^2+15t+19).
\end{align*}
If we treat $Q$ as a point at infinity on the curve $\cal{C}_{1}$
and use the method described in \cite[page 77]{Mor} we conclude that
$\cal{C}_{1}$ is birationally equivalent over $\Q(t)$ to the
elliptic curve with the Weierstrass equation
\begin{equation*}
\cal{E}_{1}:\;Y^2=X^3+f(t^2+3t)X+g(t^2+3t),
\end{equation*}
where
\begin{align*}
&f(u)=-27(1053u^4+10152u^3+37530u^2+62616u+39673),\\
&g(u)=54(9u^2+60u+85)(45u^2+192u+227)(63u^2+312u+397).
\end{align*}
The mapping $\phi:\;\cal{C}_{1}\ni (v,s) \mapsto (X,Y)\in
\cal{E}_{1}$ is given by
\begin{align*}
&v=\frac{2Y-27d(t)-6(3t^2+15t+19)(X-9c(t))}{12(X-9c(t))(3t^2+6t+4)},\\
&s=\frac{-(2Y-27d(t))^2+4(2X+9c(t))(X-9c(t))^2}{72(3t^2+6t+4)(X-9c(t))^2}.
\end{align*}

The discriminant of $\cal{E}_{1}$ is
\begin{align*}
2^{8}3^{16}&(3+3t+t^2)^2(1+3t+3t^2)^2(4+6t+3t^2)^2\times\\
&(7+9t+3t^2)^2(13+12t+3t^2)^2(19+15t+3t^2)^2,
\end{align*}
so that $\cal{E}_{1,t}$ is singular for the values $t\in \cal{A}$,
where
\begin{equation*}
\cal{A}=\{\frac{-15\pm\sqrt{-3}}{6},\frac{-9\pm\sqrt{-3}}{6},\frac{-3\pm\sqrt{-3}}{6},\frac{-6\pm\sqrt{-3}}{3},\frac{-3\pm\sqrt{-3}}{3},\frac{-3\pm\sqrt{-3}}{2}\}.
\end{equation*}
For $t\in\cal{A}$, the decomposition is of Kodaira classification
type $I_{2}$. Let us note that $\cal{E}_{1}$ is a K3-surface. As we
know, the N\'{e}ron-Severi group over $\C$, denoted by
$\op{NS}(\cal{E}_{1})=\op{NS}(\cal{E}_{1},\C)$, is a finitely
generated $\Z$-module. From Shioda \cite{Shi}, we have
\begin{equation*}
\op{rank}\op{NS}(\cal{E}_{1},\C)=\op{rank}\cal{E}_{1}(\C(t))+2+\sum_{\nu}(m_{\nu}-1),
\end{equation*}
where the sum ranges over all singular fibers of the pencil
$\cal{E}_{1,t}$, with $m_{\nu}$ the number of irreducible components
of the fiber. Here, we have
\begin{equation*}
\op{rank}\op{NS}(\cal{E}_{1},\C)=\op{rank}\cal{E}_{1}(\C(t))+2+6\cdot2\cdot(2-1).
\end{equation*}
Since the rank of the N\'{e}ron-Severi group of a K3-surface cannot
exceed 20, then the $\op{rank}\cal{E}_{1}(\C(t))\leq 6$. Although it
would be interesting to know the rank of $\cal{E}_{1}(\C(t))$
precisely, we are rather interested with the rank of
$\cal{E}_{1}(\Q(t))$. We will show that
$\op{rank}\cal{E}_{1}(\Q(t))\geq 1$.

Let us note that the curve $\cal{E}_{1}$ has three points of order
two
\begin{align*}
&T_{1}=(6(9(t^2+3t)^2+60(t^2+3t)+85),\;0),\\
&T_{2}=(3(45(t^2+3t)^2+192(t^2+3t)+227),\;0),\\
&T_{3}=(-3(63(t^2+3t)^2+312(t^2+3t)+397),\;0).
\end{align*}

On the curve $\cal{E}_{1}$ we have also the point $P=(X_{P},Y_{P})$,
where
\begin{align*}
&X_{P}=-3(9t^4+108t^3+357t^2+468t+229),\\
&Y_{P}=54(3t^2+6t+4)(3t^2+9t+7)(3t^2+15t+19)).
\end{align*}
It is easy to see that the point $P$ is of infinite order on the
curve $\cal{E}_{1}$. In order to prove this let us consider the
curve $\cal{E}_{1,\;1}$ which is specialization of the curve
$\cal{E}_{1}$ at $t=1$. We have $\cal{E}_{1,\;1}:
Y^2=X^3-48867651X+115230640770$. On the curve $\cal{E}_{1,\;1}$ we
have the point $P_{1}=(-3513,\;493506)$ which the specialization of
the point $P$ at $t=1$. Now, let us note that
\begin{equation*}
3P_{1}=\Big(\frac{3953140143}{1408969},\;
\frac{24183154596042}{1672446203}\Big).
\end{equation*}
As we know, the points of finite order on the elliptic curve
$y^2=x^3+ax+b,\;a,\;b\in\Z$ have integer coordinates \cite[page
177]{Sil}, while $3P_{1}$ is not integral point; therefore, $P_{1}$
is not a point of finite order on $\cal{E}_{1,\;1}$, which means
that $P$ is not a point of finite order on $\cal{E}_{1}$. Therefore,
$\cal{E}_{1}$ is a curve of positive rank. In particular we have
proved that $1\leq \op{rank}\cal{E}_{1}(\Q(t))\leq
\op{rank}\cal{E}_{1}(\C(t))\leq 6$. However, we have not identified
the $\Q(t)$ rank of $\cal{E}_{1}$ exactly. Numerical calculations
give that for many specializations of $\cal{E}_{1}$ at $t\in\Q$ we
have that the rank of $\cal{E}_{1,t}(\Q)$ is equal to one and this
suggest that $\op{rank}\cal{E}_{1}(\Q(t))=1$.

Before we prove that the set of rational points on the surface
$\cal{S}_{1}$ is dense in Euclidean topology we prove Zariski
density of the set of rational points. Because the curve
$\cal{E}_{1}$ is of positive rank over $\Q(t)$, the set of
multiplicities of the point $P$ i.e. $mP=(X_{m}(t),\;Y_{m}(t))$ for
$m=1,2,\ldots$\;, gives infinitely many rational curves on the curve
$\cal{E}_{1}$. Now, if we look on the curve $\cal{E}_{1}$ as on the
elliptic surface in the space with coordinates $(X,Y,t)$ we can see
that each rational curve $(X_{m},Y_{m},t)$ is included in the
Zariski closure, say $\cal{R}$, of the set of rational points on
$\cal{E}_{1}$. Because this closure consists of only finitely many
components, it has dimension two, and as the surface $\cal{E}_{1}$
is irreducible, $\cal{R}$ is the whole surface. Thus the set of
rational points on $\cal{E}_{1}$ is dense in the Zariski topology
and the same is true for the surface $\cal{S}_{1}.$

To obtain the statement of our theorem, we have to use two beautiful
results: Hurwitz Theorem (\cite[p. 78]{Sko})and the Silverman's
Theorem (\cite[p. 368]{Sil}). Let us recall that Hurwitz Theorem
states that if an elliptic curve $E$ defined over $\Q$ has positive
rank and one torsion point of order two (defined over $\Q$) then the
set $E(\Q)$ is dense in $E(\R)$. The same result holds if $E$ has
three torsion points of order two under assumption that we have
rational point of infinite order which lives on bounded branch of
the set $E(\R)$.

Silverman Theorem states that if $\cal{E}$ is an elliptic curve
defined over $\Q(t)$ with positive rank, then for all but finitely
many $t_{0}\in\Q$, the curve $\cal{E}_{t_{0}}$ obtained from the
curve $\cal{E}$ by the specialization $t=t_{0}$ has positive rank.
From this result we see that for all but finitely many $t\in\Q$ the
elliptic curve $\cal{E}_{1,t}$ is of a positive rank.

In order to finish the proof of our theorem let us define the
polynomial $X_{i}(t)=X-\mbox{coordinate of the torsion
point}\;T_{i}$ for $i=1,2,3$. Tieing these two cited theorems
together and the fact that we have inequalities
$X_{3}(t)<X_{P}(t)<X_{1}(t)<X_{2}(t)$ for each $t\in\R$ we conclude
that for all but finitely many $t$ the set $\cal{E}_{1,t}(\Q)$ is
dense in the set $\cal{E}_{1,t}(\R)$. This proves that the set
$\cal{E}_{1}(\Q)$ is dense in the set $\cal{E}_{1}(\R)$ in Euclidean
topology. Thus, the set $\cal{S}_{1}(\Q)$ is dense in the set
$\cal{S}_{1}(\R)$ in the Euclidean topology.
\end{proof}

Using the above result we can deduce the following
\begin{thm}\label{thm2}
There exists a polynomial $k\in\Z[t]$ with such a property that on
the elliptic curve $\cal{E}:\;y^2=x^3+k(t)$ there are four
independent $\Q(t)$-rational points in arithmetic progression.
\end{thm}
\begin{proof}
 In order to construct a polynomial $k\in\Z[t]$ with such a
property that for all $t\in\Z$ value $k(t)$ is an integer of AP4
type we use the points $P$ and $T_{1}$ we have constructed in the
proof of Theorem \ref{thm1}. Let us note that
\begin{equation*}
P+T_{1}=(3(99t^4+432t^3+795t^2+684t+251,-486(t^2+3t+3)(3t^2+3t+1)(3t^2+6t+4)).
\end{equation*}
Let us note that this point leads us to the point
\begin{equation*}
\phi(P+T_{1})=\Big(\frac{3t^2+9t+10}{3(2t+3)},\;-\frac{(1+3t+3t^2)(18t^4+162t^3+516t^2+738t+413)}{9(2t+3)^2}\Big),
\end{equation*}
which is the point on the curve $\cal{C}_{1}$. Performing all
necessary simplifications we find that the point $\phi(P+T_{1})$
lead us to the polynomial solution of the equation defining the
surface $\cal{S}_{1}$ in the form
\begin{equation*}
\begin{array}{ll}
  p=18t^4+54t^3+30t^2-72t-73, &\;  q=18t^4+90t^3+192t^2+210t+107, \\
  r=18t^4+126t^3+354t^2+456t+233, &\;  s=18t^4+162t^3+516t^2+738t+413.
\end{array}
\end{equation*}

Using this parametric solution and the remark on the beginning of
our proof we find the curve
\begin{equation}\label{R5}
\cal{E}:\;y^2=x^3+k(t),
\end{equation}
where
\begin{equation*}
k(t)=-324^2(2t+3)^2(3t^2+9t+10)^2(6t^2+18t+17)^2h(t),
\end{equation*}
and
\begin{equation*}
h(t)=(t^2+3t)^4+612(t^2+3t)^3+300(t^2+3t)^2-3504(t^2+3t)-5329.
\end{equation*}
On the curve $\cal{E}$ we have four points in arithmetic progression
\begin{align*}
&P_{1}=(tp(t),\;p(t)(18t^4+54t^3+30t^2-72t-73)),\\
&P_{2}=((t+1)p(t),\;p(t)(18t^4+90t^3+192t^2+210t+107)),\\
&P_{3}=((t+2)p(t),\;p(t)(18t^4+126t^3+354t^2+456t+233)),\\
&P_{4}=((t+3)p(t),\;p(t)(18t^4+162t^3+516t^2+738t+413)),
\end{align*}
where $p(t)=108(2t+3)(3t^2+9t+10)(6t^2+18t+17)$.

We will show that the above points are independent in the group
$\cal{E}(\Q(t))$ of all $\Q(t)$-rational points on the curve
$\cal{E}$. We specialize the curve $\cal{E}$ at $t=1$ and we get the
elliptic curve $\cal{E}_{1}$ given by the equation
\begin{equation*}
\cal{E}_{1}:\;y^2=x^3-111610206808689600.
\end{equation*}
On the curve $\cal{E}_{1}$ we have the points
\begin{equation*}
\begin{array}{ll}
  P_{1,\;1}=(487080,\; 62833320) &  P_{2,\;1}=(974160,\;901585080) \\
  P_{3,\;1}=(1461240,\;1734491880) &  P_{4,\;1}=(1948320,\; 2698910280),
\end{array}
\end{equation*}
which are specialization of the points $P_{1},P_{2},P_{3},P_{4}$ at
$t=1$. Using now program {\sc APECS} \cite{Con} we obtain that the
determinant of the height matrix of the points
$P_{1,\;1},\;P_{2,\;1},\;P_{3,\;1},\;P_{4,\;1}$ is equal to
266.618020487005. This proves that the points
$P_{1,\;1},\;P_{2,\;1},\;P_{3,\;1},\;P_{4,\;1}$ are independent on
the curve $\cal{E}_{1}$ and thus we get that that the points
$P_{1},P_{2},P_{3},P_{4}$ are independent on the curve $\cal{E}$.
\end{proof}

\begin{rem}
{\rm Let us consider the polynomial $g(t,x)=x^3+k(t)$, where
$k\in\Z[t]$ is defined in the proof of the above theorem. Then in
order to find a rational value of $t$ which lead to five points in
arithmetic progression on the curve $\cal{E}_{t}$ we must have
$g(t,(t-1)p(t))=\Box$ or $g(t,(t+4)p(t))=\Box$. In other words we
must be able to find rational point on one of the hyperelliptic
curves of genus 3
\begin{align*}
C_{1}:\;v^2=324&t^8+648t^7-4428t^6-21384t^5-34884t^4+\\
               &-17388t^3+12828t^2+12804t-791,\\
C_{2}:\;v^2=324&t^8+7128t^7+63612t^6+309096t^5+912816t^4\\
                &+1704132t^3+1985988t^2+1332624t+397009.
\end{align*}
Unfortunately, we are unable to find rational point on any of this
curves which lead to nonzero value of $k(t)$.

}

\end{rem}

\section{Rational points in arithmetic progressions on $y^2=x^n+k$ for $n\geq 4$}\label{sec3}

In this section we consider a natural generalization of the problem
we have considered in Section \ref{sec1}. To be more precise we
consider the following

\begin{ques}\label{ques1sec3}
Let $n\in\N$ be fixed and suppose that $n\geq 4$. It is possible to
find an integer $k$ with such a property that on the hyperelliptic
curve $y^2=x^n+k$ there are at least four rational points in
arithmetic progression?
\end{ques}

We show that if $n$ is odd then the answer on the above question is
affirmative and it is possible to find infinitely many demanded
$k$'s.

In the case when $n$ is even we show that it is possible to
construct infinitely many $k$'s with such a property that on the
curve $y^2=x^n+k$ there is three term arithmetic progression of
rational points with $x$-coordinates of the form $a,\;3a,\;5a$,
where $a$ is positive. Using the involution $(x,y)\mapsto (-x,y)$ we
can see that on these curves we will have six rational points in
arithmetic progression with $x$-coordinates of the form
$-5a,\;-3a,\;-a,\;a,\;3a,\;5a$.

We start with the following

\begin{thm}\label{thm1sec3}
Let us fix an $n\geq 2$. Then there are infinitely many integers $k$
with such a property that on the hyperelliptic curve
$H:\;y^2=x^{2n+1}+k$ there are four points in arithmetic
progression.
\end{thm}
\begin{proof}
In order to prove our theorem let us consider a hyperelliptic curve
with the equation $y^2=ax^{2n+1}+bx^2+cx+d=:f(x)$, where
\begin{equation}\label{sys1sec3}
\begin{array}{lll}
  a=\frac{-p^2+3q^2-3r^2+s^2}{2^{2n+1}-2}, & & b=\frac{p^2-2q^2+r^2}{2}, \\
  c=\frac{-(2^{2n} - 2)p^2-3q^2+(2^{2n}+2)r^2-s^2}{2^{2n+1}-2},& & d=q^2.
\end{array}
\end{equation}
For $a,b,c,d$ defined in this way we have that the points
$(-1,p),\;(0,q),\;(1,r),\;(2,s)$ are on the curve and are in
arithmetic progression. Now, let us note that if we are able to show
that the system of equations in the variables $p,q,r,s$ given by
\begin{equation}\label{sys2sec3}
\begin{cases}
p^2-2q^2+r^2=0,\\
-(2^{2n} - 2)p^2-3q^2+(2^{2n}+2)r^2-s^2=0,
\end{cases}
\end{equation}
has infinitely many rational solutions, then the points
\begin{equation*}
(-a,\;pa^{n}),\quad (0,\;qa^{n}),\quad (a,\;ra^{n}),\quad
(2a,\;sa^{n})
\end{equation*}
will be in arithmetic progression on the curve
$y^2=x^{2n+1}+da^{2n}$, where $a,b$ are given by (\ref{sys1sec3}).

Now, we show that the system (\ref{sys2sec3}) has infinitely many
solutions in rational numbers. In order to do this let us
parametrize all solutions of the first equation in the system
(\ref{sys2sec3}). Using standard method we find parametrization
given by
\begin{equation*}
p=2u^2-4uv+v^2,\quad q=2u^2-2uv+v^2,\quad r=-2u^2+v^2.
\end{equation*}
Putting now the calculated values for $p,q,r$ into the second
equation of the system (\ref{sys2sec3}) and taking $u=1$ we get
\begin{equation*}
\cal{C}^{o}_{n}:\;s^2=v^4+4(2^{2n+1}-1)v^3-8(3\cdot
2^{2n}-1)v^2+8(2^{2n+1}-1)v+4.
\end{equation*}
The curve $\cal{C}^{o}_{n}$ is a quartic curve with rational point
$Q=(0,-2)$. If we treat $Q$ as a point at infinity on the curve
$\cal{C}$ and use the method described in \cite[page 77]{Mor} one
more time we conclude that $\cal{C}^{o}_{n}$ is birationally
equivalent over $\Q$ to the elliptic curve with the Weierstrass
equation
\begin{equation*}
\cal{E}^{o}_{n}:\;Y^2=X^3-27(3\cdot2^{4n+2}+1)X+54(9\cdot
2^{4n+2}-1).
\end{equation*}

The mapping $\phi:\;\cal{E}^{o}_{n}\ni (X,Y) \mapsto (v,s)\in
\cal{C}^{o}_{n}$ is given by
\begin{align*}
&v=\frac{2Y - 27\cdot 2^{2n+2}(2^{4n+2}-1)}{6(X-3(3\cdot2^{4n+2}-1))}-(2^{2n+1}-1),\\
&s=-(v+2^{2n+1}-1)^2+\frac{2X+3(3\cdot2^{4n+2}-1)}{9}.
\end{align*}
Inverse mapping $\psi:\;\cal{C}^{o}_{n}\ni (v,s) \mapsto (X,Y)\in
\cal{E}^{o}_{n}$ has the form
\begin{align*}
&X=\frac{3}{2}(3v^2+6(2^{2n+1}-1)v+3s-4(3\cdot2^{2n}-1)),\\
&Y=\frac{27}{2}(v^3+3(2^{2n+1}-1)v^2
-4(3\cdot2^{2n}-1)v+(2^{2n+1}-1)s+2(2^{2n+1}-1) ).
\end{align*}
Let us note that the curve $\cal{E}^{o}_{n}$ has three points of
order two
\begin{equation*}
T_{1}=(6,\;0),\quad T_{2}=(3(3\cdot2^{2n+1}-1),\;0),\quad
T_{3}=(-3(3\cdot2^{2n+1}+1),\;0).
\end{equation*}
On the curve $\cal{E}^{o}_{n}$ we have also the point $P_{n}$ given
by
\begin{equation*}
P_{n}=(-3(3\cdot2^{2n+1}-5),\;54(2^{2n+1}-1)).
\end{equation*}
It is easy to see that the point $P_{n}$ is of infinite order on the
curve $\cal{E}^{o}_{n}$. In order to prove this we compute the point
$4P_{n}=(X,Y)$. We have where
\begin{align*}
&X=\frac{3\cdot2^{-4n-2}(3+3\cdot2^{16n}-2^{4n+4}+13\cdot2^{8n+1}+2^{12n+5}}{(2^{4n}-1)^2},\\
&Y=\frac{3\cdot2^{-2n-1}(3\cdot2^{8n}+1)}{2^{4n}-1}X+27\cdot2^{2n}(2^{4n}-1).
\end{align*}
It is easy to see that under our assumption $n\geq 2$ the
$X$-coordinate of the point $4P_{n}$ is not an integer. Thus,
theorem of Nagell and Lutz implies that the point $P_{n}$ is of
infinite order. This implies that the set of rational solutions of
the system (\ref{sys2sec3}) is infinite and thus our theorem is
proved.
\end{proof}

\begin{rem}
{\rm Using {\tt mwrank} program we calculated the rank $r_{n}$ of
elliptic curve $\cal{E}^{o}_{n}$ and generators for the free part of
the group $\cal{E}^{o}_{n}(\Q)$ for $2\leq n\leq 8$.  Results of our
computations are given below.}

\begin{equation*}
\begin{array}{l|l|l}
  n & r_{n} & \mbox{Generators for the free part of the group}\;\cal{E}^{o}_{n}(\Q)  \\
  \hline
  2 & 1 & (303, 17820) \\
  3 & 1 & (1167, 6966) \\
  4 & 2 & (4623, 27702),\; (11773, 1175552)\\
  5 & 2 & (18447, 110646),\; (350011167/6241,6184493104374/493039)\\
  6 & 1 & (73743, 442422) \\
  7 & 3 & (294927, 1769526),\; (153394089/400, 1214402001813/800),\\
    &   & (124356529/256, 1101957449705/4096) \\
  8 & 1 & (1179663, 7077942) \\
  \hline
\end{array}
\end{equation*}
\end{rem}
Form the above theorem we get an interesting

\begin{cor}\label{cor1sec3}
Let $n\geq 2$ and consider the family of hyperelliptic curves given
by the equation $C_{k}:\;y^2=x^{2n+1}+k$. Then the set $\cal{A}$ of
integers $k$ with such a property that on the curve $C_{k}$ there
are at least 8 rational points, is infinite. Moreover, we can
construct the set $\cal{A}$ that for each pair
$k_{1},\;k_{2}\in\cal{A}$ the curves $C_{k_{1}},\;C_{k_{2}}$ are not
isomorphic.
\end{cor}
\begin{proof}
First part of our theorem is an immediate consequence of the
previous theorem. Second part of our theorem is a simple consequence
of the following reasoning. Curves $C_{k_{1}}$ and $C_{k_{2}}$ are
isomorphic if and only if $k_{1}/k_{2}\in\Q^{10}$. From the previous
theorem we can take $k=q^2=(v^2-2v+2)^2$ for some $v\in\Q$ which is
calculated from the point which lies on the elliptic curve
$\cal{C}^{o}_{n}$. Let us suppose that we constructed the integers
$k_{1},\;k_{2},\;\ldots,\;k_{n}$ such that the curves $C_{k_{i}}$
are pairwise non-isomorphic. We have that
$k_{i}=q_{i}^{2}=(v_{i}^2-2v_{i}+2)^2$ for $i=1,\;2,\ldots,\;n$.
Then the $n$ curves $(2v^2-2v+2)^2=(v_{i}^2-2v_{i}+2)^2w^{10}$ for
$i=1,2\;\ldots,\;n$ are all of genus $\geq 2$, thus the set of
$\C_{1}(\Q)\cup\ldots\cup C_{n}(\Q)$ is finite (Faltings Theorem
\cite{Fal}). Because the elliptic curve $\cal{C}^{o}_{n}$ has
infinitely many rational points we can find $v_{n+1}$ such that the
curve $C_{k_{n+1}}$ with $k_{n+1}=(v_{n+1}^2-2v_{n+1}+2)^2$ is not
isomorphic to any of the curves $C_{i}$ for $i=1,2\ldots,\;n$. Using
now the presented reasoning we can construct an infinite set
$\cal{A}$ with demanded property.
\end{proof}

The above corollary and the Corollary \ref{cor2sec3} gives a
generalization A. Bremner result from \cite{Bre} (Theorem 2.1 and
Theorem 3.1) .

Theorem \ref{thm1sec3} we have proved suggest the following

\begin{ques}\label{ques2sec3}
Let us fix an integer $n\geq 1$. What is the smallest value
$|k_{2n+1}|\in\N$, say $M_{2n+1}$, with such a property that on the
curve $y^2=x^{2n+1}+k_{2n+1}$ there are at least four rational
points in arithmetic progression?
\end{ques}

\begin{exam}
{\rm If we take $n=2$ then from the existence of rational point of
infinite order on the curve $\cal{E}^{o}_{2}$ we get that
\begin{equation*}
M_{5}\leq
3391541395170708368688169980^4\cdot2609^2\cdot127165689041^2.
\end{equation*}
}
\end{exam}

Moreover in the light of Corollary \ref{cor1sec3} it is natural to
sate the following

\begin{ques}\label{ques3sec3}
Let us fix an integer $n\geq 2$ and consider the hyperelliptic curve
$H_{2n+1}:\;y^2=x^{2n+1}+k_{2n+1}$ where $k_{2n+1}$ is constructed
with the use method we have presented in the proof of the previous
theorem. In particular on the curve $H_{2n+1}$ we have four points
in arithmetic progression, say $P_{i}$ for $i=1,2,3,4$. Are the
divisors $(P_{i})-(\infty)$ independent in the jacobian variety
associated with the curve $H_{2n+1}$?
\end{ques}

Now, we prove the following
\begin{thm}\label{thm2sec3}
Let us fix an $n\geq 2$. Then there are infinitely many integers $k$
with such a property that on the hyperelliptic curve
$H:\;y^2=x^{2n}+k$ there are six points in arithmetic progression.
\end{thm}
\begin{proof}
In order to prove our theorem let us consider a hyperelliptic curve
with the equation $y^2=x^{2n}+ax^2+bx+c=:f(x)$, where
\begin{align*}
&a=\frac{p^2-2q^2+r^2-5^{2n}+2\cdot3^{2n}-1}{8},\\
&b=\frac{-2p^2+3q^2-r^2+5^{2n}-3^{2n+1}+2}{2},\\
&c=\frac{15p^2-10q^2+3r^2-3\cdot5^{2n}+10\cdot3^{2n}-15}{8}.
\end{align*}
Then we have that $f(1)=p^2,\;f(3)=q^2,\;f(5)=r^2$. It is easy to
see that in order to prove our theorem it is enough to find
infinitely many solutions of the system of equations $a=b=0$ in
rational numbers $p,q,r$ . Indeed, if $a=b=0$, then on the curve
$y^2=x^{2n}+c$ we will have six points in arithmetic progression
which $x$-coordinates belong to he set $\{-5,-3,-1,1,3,5\}$.

System $a=b=0$ is equivalent with the system of equations
\begin{equation}\label{sys3sec3}
\begin{cases}
q^2=p^2+3^{2n}-1,\\
r^2=p^2+5^{2n}-1.
\end{cases}
\end{equation}
By putting $p=u+1,\;q=tu+3^n$ we find that all rational solutions of
the equation $q^2=p^2+3^{2n}-1$ are contained in the formulas
\begin{equation*}
p=\frac{t^2-2\cdot3^{n}t+1}{t^2-1},\quad\quad
q=-\frac{3^{n}t^2-2t+3^{n}}{t^2-1}.
\end{equation*}
Putting the value of $p$ we have calculated into the second equation
of the system (\ref{sys3sec3}) we get the equation of the quartic
curve
\begin{equation*}
\cal{C}^{e}_{n}:\;s^2=5^{2n}t^{4}-4\cdot
3^{n}t^{3}-2(5^{2n}-2\cdot3^{2n}-2)t^2-4\cdot 3^{n}t+5^{2n}=:g(t),
\end{equation*}
where $s=(t^2-1)r$. For convenience of our calculations let us put
$u=3^n$ and $v=5^n$. Then, the polynomial $g$ takes the form
$g_{u,v}(t)=v^2t^4-4ut^3-2(v^2-2u^2-2)t^2-4u+v^2$.

The curve $\cal{C}^{e}_{n}: s^2=g_{u,v}(t)$ is a quartic curve with
rational point $Q=(0,v)$. If we treat $Q$ as a point at infinity on
the curve $\cal{C}^{e}_{n}$ and use the method described in
\cite[page 77]{Mor} one more time we conclude that $\cal{C}^{e}_{n}$
is birationally equivalent over $\Q$ to the elliptic curve with the
Weierstrass equation
\begin{align*}
\cal{E}^{e}_{n}:\;Y^2=X^3-27(v^4-&(u^2+1)v^2+u^4-u^2+1)X +\\
                                 &+27(1+u^2-2v^2)(2u^2-v^2-1)(u^2+v^2-2).
\end{align*}
The mapping $\phi:\;\cal{E}^{e}_{n}\ni (X,Y) \mapsto (t,x)\in
\cal{C}^{e}_{n}$ is given by
\begin{align*}
&t=\frac{v^3Y-27u(u^2-v^2)(v^2-1)}{3v^2(v^2X-3(3u^2-2v^2+v^4-2u^2v^2)}+\frac{u}{v^2},\\
&s=-\frac{1}{v^3}\Big(v^2t-\frac{u}{v^2}\Big)^2+\frac{v^2X+9u^2-6(u^2+1)v^2+3v^4}{9v^3}.
\end{align*}
Inverse mapping $\psi:\;\cal{C}^{e}_{n}\ni (t,s) \mapsto (X,Y)\in
\cal{E}^{e}_{n}$ has the form
\begin{align*}
&X=\frac{2-6tu+2u^2+3sv+(3t^2-1)v^2}{2},\\
&Y=-\frac{27}{2}(su+((3t^2+1)u-2t(u^2+1))v-stv^2-(t^3-t)v^3).
\end{align*}
Let us note that the curve $\cal{E}^{e}_{n}$ has three points of
order two
\begin{equation*}
T_{1}=(3(1+u^2-2v^2),\;0),\quad T_{2}=(3(u^2+v^2-2),\;0),\quad
T_{3}=(-3(1-2u^2+v^2),\;0).
\end{equation*}
On the curve $\cal{E}^{e}_{n}$ we have also the point $P_{n}$ given
by
\begin{equation*}
P_{n}=(3(u^2+v^2+1),\;-27uv).
\end{equation*}
It is easy to see that the point $P$ is of infinite order on the
curve $\cal{E}^{2}_{n}$. Indeed, if we calculate
$2P_{n}=(X_{2},Y_{2})$ then we have
\begin{equation*}
X_{2}=\frac{3(3u^4-2(u^2+1)u^2v^2+(3-2u^2+3u^4)v^4)}{4u^2v^2}.
\end{equation*}
Due to the fact that $u=3^n$ and $v=5^n$ it is easy to see that for
any choice for $n\geq 2$ the above fraction is not an integer. From
the Nagell-Lutz theorem we get that the point $P_{n}$ is not of
finite order on the curve $\cal{E}^{e}_{n}$. This implies that the
set of rational solutions of the system (\ref{sys3sec3}) is
infinite. As a consequence we get that for any $n\geq 2$ we can
construct infinitely many $k$'s with such a property that on the
curve $y^2=x^{2n}+k$ we have three points in arithmetic progression
with $x$-coordinates belonging to the set $\{m,\;3m,\;5m\}$ with
$m>0$. Note that on this curve we have also rational points with
$x$-coordinates belonging to the set $\{-m,\;-3m,\;-5m\}$. This
observation finishes the proof of our theorem.

\end{proof}
\begin{rem}
{\rm Using {\tt mwrank} program we calculated the rank $r_{n}$ of
elliptic curve $\cal{E}^{e}_{n}$ and generators for the free part of
the group $\cal{E}^{e}_{n}(\Q)$ for $2\leq n\leq 10$.  Results of
our computations are given below.}

\begin{equation*}
\begin{array}{l|l|l}
  n & r_{n} & \mbox{Generators for the free part of the group}\;\cal{E}^{e}_{n}(\Q)  \\
  \hline
  2 & 1 & (3840, 176256) \\
  3 & 2 & (80808, 14478912),\;(130704, 40007520) \\
  4 & 2 & (1230432, 116328960),\; (31769376/25, 24804389376/125)\\
  5 & 1 & (36278088, 68748343488)\\
  6 & 1 & (836384640, 5022400795776)\\
  7 & 1 & (19863352968, 370669722011712)\\
  8 & 1 & (480955252992, 27480236025415680)\\
  \hline
\end{array}
\end{equation*}
\end{rem}

Using similar argument as in the proof of the Corollary
\ref{cor1sec3} we can prove the following

\begin{cor}\label{cor2sec3}
Let $n\geq 2$ and consider the family of hyperelliptic curves given
by the equation $C_{k}:\;y^2=x^{2n}+k$. Then the set $\cal{A}$ of
integers $k$ with such a property that on the curve $C_{k}$ there
are at least 12 rational points, is infinite. Moreover, we can
construct the set $\cal{A}$ that for each pair
$k_{1},\;k_{2}\in\cal{A}$ the curves $C_{k_{1}},\;C_{k_{2}}$ are not
isomorphic.
\end{cor}

Similarly as in the case of odd exponents we can state the following
questions.

\begin{ques}\label{ques4sec3}
Let us fix an integer $n\geq 2$. What is the smallest value
$|k_{2n}|\in\N$, say $M_{2n}$ with such a property that on the curve
$y^2=x^{2n}+k_{2n}$ there are at least three rational points in
arithmetic progression?
\end{ques}

\begin{ques}\label{ques5sec3}
Let us fix an integer $n\geq 2$ and consider the hyperelliptic curve
$H_{2n}:\;y^2=x^{2n}+k_{2n}$ where $k_{2n}$ is constructed with the
use method we have presented in the proof of the above theorem. In
particular, on the curve $H_{2n}$ we have three points in arithmetic
progression, say $P_{i}$ for $i=1,2,3$. Do the divisors
$(P_{i})-(\infty)$ are independent in the jacobian variety
associated with the curve $H_{2n}$?
\end{ques}

\section{Sextic threefold related to five rational points in arithmetic progressions on $y^2=x^3+k$}\label{sec4}

In the section \ref{sec2} we have used very natural reasoning in
order to construct quartic surface which is closely related to the
problem of existence of numbers of AP4 type. A natural question
arises whether we can construct an algebraic variety, say $\cal{T}$,
with such a property that each (nontrivial) rational point on
$\cal{T}$ gives an integer $k$ of AP5 type, so an integer $k$ such
that on the curve $y^2=x^3+k$ we have five rational points in
arithmetic progression. In order to construct demanded hypersurface
we will use similar method as in section \ref{sec2}.

Let $f(x,y)=y^2+ay-(bx^3+cx^2+dx+e)\in\Q[a,b,c,d,e][x]$ and consider
the curve $C:\;f(x,y)=0$. Using now change of coordinates
\begin{equation*}
(x,\;y)=\Big(\frac{X-12c}{36b},\frac{Y-108ab}{216b}\Big)\;
\mbox{with inverse}\; (X,\;Y)=(12(c+3bx), 108b(a+2y)),
\end{equation*}
we can see that the curve $C$ is birationally equivalent to the
curve
\begin{equation*}
E:\;Y^2=X^3-432(c^2-3bd)X+432(27a^2b^2+8c^3-36bcd+108b^2e).
\end{equation*}

Now, let $p,q,r,s,t$ be free parameters and consider the system of
equations
\begin{equation*}
f(-2,p)=f(-1,q)=f(0,r)=f(1,s)=f(2,t)=0.
\end{equation*}
This system has exactly one solution in respect to $a,b,c,d,e$. This
solution belong to the field $\Q(p,q,r,s,t)$ and has the form
\begin{equation*}
a=\frac{A}{6H},\quad b=\frac{B}{6H},\quad c=\frac{C}{6H},\quad
D=\frac{D}{6H},\quad e=\frac{E}{6H},
\end{equation*}
where $A,\ldots,E\in\Z[p,q,r,s,t]$ are homogeneous, $A$ is of degree
two and $B,\ldots,E$ are of degree three. Moreover, we have
$H=p-4q+6r-4s+t$. From this computations we can see that in order to
find an integer $k$ with such a property that on the curve
$y^2=x^3+k$ there are five points in arithmetic progression it is
enough to find rational points on the sextic threefold given by the
equation
\begin{equation*}
\cal{T}:\;C(p,q,r,s,t)^2=3B(p,q,r,s,t)D(p,q,r,s,t),
\end{equation*}
where
\begin{equation*}
\begin{array}{ll}
  B= & (p-3q+3r-s)t^2-(p^2-3q^2+3r^2-s^2)t+(q-3r+3s)p^2+\\
     & -(q^2-3r^2+3s^2)p+2(q-s)(3qr-3r^2-4qs+3rs),\\
     & \\
  C= & -3((t^2+p^2)(q-2r+s)-(t+p)(q^2-2r^2+s^2)+2r(q^2-qr-rs+s^2)), \\
     & \\
  D= & -(p-6q+3r+2s)t^2+(p^2-6q^2+3r^2+2s^2)t+(2q+3r-6s)p^2+\\
     &-(2q^2+3r^2-6s^2)p-8(q-s)(3qr-3r^2-4qs+3rs).
\end{array}
\end{equation*}

We have obvious automorphisms of order two which act on the
$\cal{T}$ given by
\begin{equation*}
(p,q,r,s,t)\mapsto(t,s,r,q,p),\quad (p,q,r,s,t)\mapsto
(-p,-q,-r,-s,-t).
\end{equation*}
By a trivial rational point on the hypersurface $\cal{T}$ we will
understand the point $(p,q,r,s,t)$ which lies on one of the lines
\begin{equation*}
L_{1}:\;p=q=r,s=t,\;L_{2}:\;p=q=s,r=t,\;\ldots,\;L_{10}:\;r=s=t,p=q,
\end{equation*}
or on the hyperplane $H:\;p-4q+6r-4s+t=0$.

We perform small numerical calculations in order to find a
non-trivial rational point on $\cal{T}$. We compute all integer
solutions of the equation which define the hypersurface $\cal{T}$
under assumption $\op{max}\{|p|,|q|,|r|,|s|,|t|\}\leq 10^2$.
Unfortunately in this range all solutions are trivial. This suggest
the following

\begin{ques}\label{ques4.1}
Is the set of non-trivial rational points on the hypersurface
$\cal{T}$ non-empty?
\end{ques}

\bigskip

\vskip 0.5cm

\hskip 4.5cm        Maciej Ulas

 \hskip 4.5cm       Jagiellonian University

 \hskip 4.5cm       Institute of Mathematics

 \hskip 4.5cm       {\L}ojasiewicza 6

 \hskip 4.5cm       30 - 348 Krak\'{o}w, Poland

 \hskip 4.5cm      e-mail:\;{\tt Maciej.Ulas@im.uj.edu.pl}

 \hskip 4.5cm      e-mail:\;{\tt maciej.ulas@gmail.com}

\end{document}